\documentclass[10pt]{amsart}
\usepackage{amsfonts}
\usepackage{amssymb}
\usepackage{graphicx}
\usepackage{newlfont}
\usepackage{mathrsfs}
\usepackage{graphics}
\usepackage{amsmath, amssymb, amsfonts,amsthm}
\usepackage{textcomp}
\usepackage{color}
\usepackage{mathtools}
\usepackage{multicol}
\usepackage{bbm}
\usepackage{float}
\usepackage[prependcaption,textsize=tiny,color=white!70]{todonotes}
\newtheorem{theorem}{Theorem}[section]

\newtheorem{lem}[theorem]{Lemma}

\newtheorem{definition}[theorem]{Definition}

\newtheorem{conjecture}[theorem]{Conjecture}

\numberwithin{equation}{section}

\newcommand{\bP}{\mathbb{P}}
\newcommand{\el}{\stackrel{d}{=}}

\newcommand{\bE}{\mathbb{E}}

\newcommand{\bone}{\mathbbm{1}}
\newcommand{\Var}{\mathrm{Var}}
\newcommand{\Cov}{\mathrm{Cov}}

\newcommand{\TT}{\mathcal{T}}
\newcommand{\LL}{\mathcal{L}}

\newcommand{\nn}{\nonumber}

\def\beq#1\eeq{\begin{equation}#1\end{equation}}
\def\le{\leq}
\def\de{\delta}
\def\De{\Delta}
\def\ge{\geq}
\def\sig{\sigma}
\def\wtilde{\widetilde}
\def\el{\stackrel{d}{=}}

\def\cgp{\stackrel{p}{\longrightarrow}}

\def\beqnn#1\eeqnn{\begin{equation*}#1\end{equation*}}
\def\beqn#1\eeqn{\begin{equation}#1\end{equation}}
\def\beqna#1\eeqna{\begin{eqnarray}#1\end{eqnarray}}

\newenvironment{romenumerate}[1][-10pt]{
\addtolength{\leftmargini}{#1}\begin{enumerate}
 }{\end{enumerate}}

\newcommand\set[1]{\ensuremath{\{#1\}}}

\newcommand\bigpar[1]{\bigl(#1\bigr)}

\newcommand\lrpar[1]{\left(#1\right)}

\newcommand\bigabs[1]{\bigl\lvert#1\bigr\rvert}

 \newcommand\gG{\Gamma}
 \newcommand\rhoqq{\rho^{1/2}}
 \newcommand\parentwn{\stackrel{\leftarrow}{w_n}}

\title[Strong Convergence of Infinite Color Balanced Urns]
        {Strong Convergence of Infinite Color Balanced Urns Under Uniform Ergodicity}
\author{Antar Bandyopadhyay} 
\address[Antar Bandyopadhyay]{Theoretical Statistics and Mathematics Unit \\
         Indian Statistical Institute, Delhi Centre \\ 
         7 S. J. S. Sansanwal Marg \\
         New Delhi 110016 \\
         INDIA}
\address{Theoretical Statistics and Mathematics Unit, 
         Indian Statistical Institute, Kolkata;
         203 B. T. Road, Kolkata 700108, INDIA}
\email{antar@isid.ac.in}

\author{Svante Janson}  
\address[Svante Janson]{Department of Mathematics, Uppsala University,
         Box 480, 751 06 Uppsala, Sweden\\}         
\email{svante@math.uu.se} 
    
\author{Debleena Thacker}  
\address[Debleena Thacker]{Department of Mathematics, Uppsala University,
         Box 480, 751 06 Uppsala, Sweden\\}         
\email{thackerdebleena@gmail.com} 
        
\date{\today}

\begin{document}

\begin{abstract}
We consider the generalization of the P\'olya urn scheme with possibly infinite many 
colors as introduced in \cite{Th-Thesis, BaTH2014, BaTh2016, BaTh2017}.
For countable many colors, we prove almost sure convergence of the urn configuration
under \emph{uniform ergodicity} assumption on the associated Markov chain. The
proof uses a stochastic coupling of the sequence of chosen colors with  
a \emph{branching Markov chain} on a weighted \emph{random recursive tree} as described in
\cite{BaTh2017, Sv_2018}. Using this coupling we estimate the covariance between any two selected colors. In particular, we reprove the limit theorem for the classical 
urn models with finitely many colors.  
\end{abstract}



\maketitle         

\section{Introduction}
\label{Sec: Intro}
 
P\'olya urn schemes and its various generalizations have been a key element of study for random processes with reinforcements.  Starting from the seminal work of P\'olya \cite{Polya30}, 
various types of urn schemes with finitely many colors have been 
widely studied in literature, see \cite{Pe07} for an extensive survey of the known 
classical results and some of the modern works can be found in
\cite{Svante1, Svante2,
BaiHu05, FlDuPu06, maulik1, maulik2, DasMau11, ChKu13, ChHsYa14}.

The P\'olya urn models with colors indexed by a general Polish space was first introduced in \cite{BlackMac73}. But, unlike in the classical case, where the set of colors
is finite, models with infinite colors was not studied in details till very recently. 
A new generalization for \emph{balanced} urn schemes with infinitely many colors 
was again introduced in 
\cite{Th-Thesis} and subsequently in the papers \cite{BaTH2014, BaTh2016, BaTh2017}. 
These work have since then generated a lot of interest and such models are
now receiving considerable attention \cite{Sv_2018, Sv_2017, Mai_Mar_2017}. 
In this paper, we will consider the 
infinite color balanced urn model, where the color set is countably infinite.

\subsection{Model}
\label{Subsec: Model} 
In this work we will consider the same generalization of the P\'olya urn scheme 
with infinite many 
colors as defined in \cite{Th-Thesis, BaTH2014, BaTh2016, BaTh2017}. However, we will
focus on the special case where the set of 
colors, is \emph{countably infinite}, which will be denoted by $S$.
We follow similar framework and 
notations as in \cite{BaTH2014, BaTh2016, BaTh2017}. For the sake of completeness,
we provide here a brief description of the model.

Let $R$ be a $S \times S$ (infinite) matrix with non-negative entries, 
representing the \emph{replacement scheme}. We will assume that $R$ is
\emph{balanced}, that is, each row sum is equal and finite. In that case, it is
customary to take $R$ a \emph{stochastic matrix} (see \cite{BaTh2017} for details).

We will denote by $U_n := \left(U_{n,v}\right)_{v \in S} \in [0, \infty)^{S}$,
the random configuration of the urn at time $n \geq 0$. We will view it as an 
infinite vector (with non-negative entries) which is in $\ell_1 \equiv \ell_1\left(S\right)$, 
and thus can also be
viewed as a (random) finite measure on $S$.
Intuitively, we will define $U_n$,   
such that, if $Z_n$ be the randomly chosen color at the $\left(n+1\right)$-th draw,
then
the conditional distribution of $Z_n$ given the ``\emph{past}'', will 
satisfy,
for all $z \in S$,
\[
\bP\left( Z_n =z\,\Big\vert\, U_n, U_{n-1}, \cdots, U_0 \right) \propto U_n\left(z\right).
\]
Formally, starting with a non-random $U_0 \in \ell_1$, we define 
$\left(U_n\right)_{n \geq 0} \subseteq \ell_1$, recursively as follows
\begin{equation}
\label{Equ:Fundamental-Recurssion}
U_{n+1}  = U_n  + R_{Z_n}, 
\end{equation}
where $R_z$ denotes the $z$-th row of the matrix $R$, and
\begin{equation}
\bP\left( Z_n =z \,\Big\vert\, U_n, U_{n-1}, \cdots, U_0 \right) = \frac{U_{n,z}}{n+t},
\label{Equ:Conditional-Distribution-of-Zn}
\end{equation}
where $U_0$ is a $\ell_1$-vector with total mass denoted by $0 < t < \infty$, that is,
$\sum_{v \in S} U_{0,v} = t \in \left(0,\infty\right)$. 

Observe that, one can now associate with such an urn model a Markov chain $\left(X_n\right)_{n \geq 0}$ on the countable state space $S$, with transition matrix $R$ and initial distribution $U_0/t$.
Conversely, given any Markov chain $\left(X_n\right)_{n \geq 0}$,
on the countable state space $S$, with transition matrix $R$ and a vector $U_0\in \ell_1$,  one can associate a balanced urn model $\left(U_n\right)_{n \geq 0}$, satisfying
equations ~\eqref{Equ:Fundamental-Recurssion} and 
~\eqref{Equ:Conditional-Distribution-of-Zn}.
We will call such a Markov chain $\left(X_n\right)_{n \geq 0}$, as a Markov chain associated with the urn model $\left(U_n\right)_{n \geq 0}$.

It has been observed in \cite{BaTh2016, Mai_Mar_2017} that the asymptotic properties 
of the urn model so defined are 
determined by the asymptotic properties of the associated Markov chain. 
In fact, in \cite{BaTh2016, Mai_Mar_2017}, the authors have shown that 
the urn sequence $\left(U_n\right)_{n \geq 0}$ has same law as that of a
\emph{branching Markov chain} with transition matrix $R$, initial distribution $U_0/t$
and defined on the \emph{random recursive tree}. In Section \ref{Sec:Coupling}, we
provide the details of this representation.

\subsection{Main Result}
\label{SubSec:Result}
 In this paper, we consider the case when $R$ is irreducible, aperiodic and positive recurrent.
From
classical theory (see Section XV.7 of \cite{FellerVol1} for the details), it is well know that, in that case, the chain has
unique stationary distribution, say, $\pi$, satisfying the equation
\beq
\label{Eq:Stationary_distribution_R}
\pi R=\pi.
\eeq
Moreover, such a chain
 is \emph{ergodic}, that is, 
for any $u, v \in S$,  
\beqnn
\lim_{n \to \infty} R^{n}(u, v)= \pi_{v},
\eeqnn
where $R^n$ is the $n$-step transition matrix, which is nothing but the 
$n$ fold composition of $R$ with itself. Note that as $S$ is countable, $R^n$ is
just the $n$ fold multiplication of $R$.

In this work we will further assume that the chain is 
\emph{uniformly ergodic}. For the
sake of completeness, we provide the definition here.
(One often uses a version with summation over $v$ in \eqref{Eq:GeoErforR};
we need only the version below.)
\begin{definition}\label{Duerg}
A Markov chain with transition matrix $R$ on a countable
state space $S$ is called \emph{uniformly ergodic}, if
there exists positive constants, $0<  \rho < 1$ and $C > 0$, such that, 
for any time $n \geq 1$ and for any states $u, v \in S$,
\begin{equation}
\big{\rvert} R^n\left(u, v\right)-\pi_{v} \big{\lvert} \le C \rho^n.
\label{Eq:GeoErforR}
\end{equation} 
\end{definition}

We note here that if $S$ is finite then an irreducible and aperiodic chain
is necessarily uniformly ergodic (see Theorem 4.9 of \cite{LevPeEli2009}). However, when $S$ is infinite (even countable) there are ergodic chains which are not
uniformly ergodic (see e.g. \cite{Haggs2005}).

Our main result is as follows:
\begin{theorem}
\label{Thm:a.s.convg_irr} 
Consider an urn model $\left(U_n\right)_{n \geq 0}$ as defined by the equations 
~\eqref{Equ:Fundamental-Recurssion} and  ~\eqref{Equ:Conditional-Distribution-of-Zn},
with colors indexed by a countably infinite set $S$, a balanced 
replacement matrix $R$, and an initial configuration $U_0$. 
We assume that $R$ is a stochastic matrix which is 
irreducible, aperiodic, positive recurrent with stationary distribution $\pi$, 
and uniformly ergodic, that is, satisfying  \eqref{Eq:GeoErforR}.
\begin{romenumerate}[-16pt]
  
\item \label{convg_irr}
Then as 
$n \to \infty$,
\begin{equation}
\frac{U_n}{n+t}\longrightarrow \pi \text{ a.s.}, 
\label{Eq:a.s.convg_irr}
\end{equation} 
where the convergence is coordinate wise and also in $\ell_1$.
\item\label{local_times_color}
For any $v \in S$, let
$N_{n,v}:=\sum_{k=0}^n \bone_{\{Z_k=v\}}$,
denote the number of times the color $v$ is chosen upto time $n$. Then, as
$n \to \infty$
\beq\label{Eq:local_time_color_almost_sure}
\frac{N_{n,v}}{n+1} \longrightarrow \pi_v \text{ a.s.},
\eeq
where the convergence is coordinate wise and also in $\ell_1$.
\end{romenumerate}
\end{theorem}

\subsection{Background and Motivation}
\label{SubSec:Back}
It is known (see, for example, Theorems 3.3(a) and 3.4(a) of \cite{BaTh2016}) 
that under our set up, as $n \to \infty$, 
\beqn\label{Eq:Irr_aperiodic_prob_convgs_urn}
\frac{U_n}{n+t} \cgp \pi,
\eeqn 
and also for any $v \in S$,
\beqn 
\label{Eq:Irr_aperiodic_expectation_convgs_urn}
\bP\left(Z_n=v\right) = \frac{\bE[U_{n,v}]}{n+t} \longrightarrow \pi_v.
\eeqn
Recall, $Z_n$ denotes the randomly chosen color at the $\left(n+1\right)$-th draw,
from the urn, when its (random) configuration is $U_n$. 
Our result strengthens this result to strong convergence. However, we would
like to point-out that the results in \cite{BaTh2016} (Theorems 3.3(a) and 3.4(a)),
only needs assumption of ergodicity for the associated Markov chain, while our main result in this work needs more stronger assumption of
uniform ergodicity of the associated Markov chain. As discussed above the two assumptions are identical when
$S$ is finite. It is worthwhile to note here that for $S$ finite our result is essentially 
the classical result for Freedman-P\'{olya}-Eggenberger  type urn models 
\cite{Gouet, BaiHu05, Svante1, AthKar68}. The classical
results mainly use three types of techniques, namely, the martingale techniques \cite{Gouet, maulik1, maulik2, DasMau11}; stochastic approximations \cite{LaPa2013} 
and embedding into continuous time pure birth processes \cite{AthKar68, Svante1, Svante2, BaiHu05}. 
Typically, the analysis of a finite color urn is heavily dependent on the 
\emph{Perron-Frobenius theory} \cite{Sene06} of matrices with positive entries
and \emph{Jordan Decomposition}
of finite dimensional matrices
\cite{AthKar68, Gouet, Svante1, Svante2, BaiHu05, maulik1, DasMau11}. 
Unfortunately, such techniques are unavailable when $S$ is infinite, even 
when countable. Our method bypasses the use of such techniques and instead uses
the newer approach developed in  \cite{BaTh2016, Mai_Mar_2017}. 
Our extra assumption (uniform
ergodicity) is needed only when $S$ infinite. Thus the result stated above re-proves
the classical result for the finite color urn model using the new technique. The result
essentially completes the work developed in \cite{BaTh2016, Mai_Mar_2017} for 
the case when $S$ is countable. We would like to note here that similar results
for a null recurrent case (when the chain is a random walk) has been 
derived in \cite{Mai_Mar_2017, Sv_2018}.

\subsection{Discussion on the assumption of uniform ergodicity}
\label{SubSec:Remarks}
As discussed above, when $S$ is finite the assumption of 
uniform ergodicity is equivalent to the assumption of ergodicity of the 
associate Markov chain \cite{LevPeEli2009}. In particular, it holds for 
irreducible and aperiodic chain.
However, when $S$ is infinite it is indeed a much
stronger assumption.  Necessary and sufficient condition under which 
a chain is uniformly ergodic can be found in \cite{Me_Tw_2009}. In particular,
an irreducible and aperiodic chain on a countable state space is uniformly ergodic, if and only if, the so called \emph{Doeblin's condition} is satisfied
(see Section 16.2 of \cite{Me_Tw_2009}). This condition is satisfied by 
many Markov chains on countable infinite state space, but it is indeed restrictive. 
We need this assumption in the proof we provide in the Section ~\ref{Sec:Proofs}. 
However, we do feel that this condition is not necessary in general. We, in fact,
make the following conjecture:
\begin{conjecture}
\label{Conj:Non-uniform-ergodicity}
Consider an urn model $\left(U_n\right)_{n \geq 0}$ as defined by the equations 
~\eqref{Equ:Fundamental-Recurssion} and  ~\eqref{Equ:Conditional-Distribution-of-Zn},
with colors indexed by a countably infinite set $S$, a balanced 
replacement matrix $R$, and an initial configuration $U_0$. 
Assume that $R$ is a stochastic matrix which is 
irreducible, aperiodic, positive recurrent with stationary distribution $\pi$. Then,  
the convergence in ~\eqref{Eq:a.s.convg_irr}
and \eqref{Eq:local_time_color_almost_sure}
holds a.s.\ and also in $\ell_1$. 
\end{conjecture}

\subsection{Outline}
\label{SubSec:Outline}
In the following section we provide some details about the representation of a
balanced urn in terms of a branching Markov chain on a weighted random recursive tree, which
is our main tool to prove Theorem ~\ref{Thm:a.s.convg_irr}. 
Section ~\ref{Sec:Proofs}  provides the proof of Theorem ~\ref{Thm:a.s.convg_irr}. 
In Section ~\ref{Sec:Appl}, we discuss a non-trivial application of our main result.

\section{Coupling of branching Markov chains and urn models}\label{Sec:Coupling}
It is known from \cite{Th-Thesis, BaTh2016, Mai_Mar_2017, Sv_2018} that the law for the entire sequence of randomly selected colors $\left(Z_n\right)_{n \ge 0}$ can be represented in terms of a \textit{branching Markov chain} on 
a \emph{random recursive trees}. For the sake of completeness, we will briefly discuss this representation here. We will later use this representation to prove the main result of the paper.

\subsection{Weighted random recursive tree}\label{Sec: WRRT}
Random recursive trees (RRT) are well studied in literature, see Ch.6 of
\cite{Drm_2009}. The weighted version for RRT has been introduced and
defined in \cite{Sv_2018}. For $n \ge -1$, let $\TT_n$ be the random
recursive tree on $n+2$ vertices, with $o$ as the root, and the other
vertices labeled as $\{w_0,w_1,\ldots, w_n\}$, where the increasing
subscripts of the vertices indicate the order in which they are
attached. The root is given some initial weight $t>0$. Every other node has
weight $1$. Initially, we start with $\TT_{-1}$ which consists only of the
root, denoted by $o$. Now we construct recursively the sequence of trees
$\left(\TT_n\right)_{n \ge -1}$, where the parent of the incoming node in
$\TT_n$ is chosen proportional to its weight, that is, the parent is the
root $o$, with probability $t/(n+t+1)$, and any other vertex with
probability $1/(n+t+1)$. Define the infinite random recursive tree as  
\beq\label{Eq:Def_Inf_RRT}
\TT:= \bigcup_{n \ge -1} \TT_n.
\eeq
\subsection{Branching Markov chain on RRT} \label{BMC}
The definition for \textit{branching Markov chain on the random recursive tree}, which we abbreviate as BMC on RRT, as discussed in the context of this paper is available in details in \cite{BaTh2016, Mai_Mar_2017}. 
To facilitate convenient reading of this paper, we discuss briefly the BMC on RRT as available in \cite{BaTh2016}.

Recall that the set of colors are indexed by a set $S$. Let $\Delta\not\in S$ be a symbol. We say a stochastic process $\left(W_n\right)_{n \ge -1}$ with state space 
$S\cup \Delta$ is a \textit{branching Markov chain on} $\TT$, starting at the root $o$ and at a position $W_{-1}=\Delta$, if for any $n \geq 0$ and for any $z \in S$
\begin{align}
  \label{Eq:transition_for_BRW_RRT}
\bP \left(W_n = z  \mid W_{n-1}, W_{n-2},  \ldots, W_{-1}; \TT_n \right)=
\begin{cases} 
U_0(z)/t, & \text{ if }\parentwn = o,\\
R\left(W_{j}, z\right) & \text{ if }\parentwn=w_j,
\end{cases}                                     
\end{align}
where $\parentwn$ is the parent of the vertex $w_n$
in $\TT_n$. Note that we here denote the vertices of $\TT_n$ as
$\left\{o,w_0,w_1 \cdots, w_n \right\}$.

\subsection{Representation Theorems}
The coupling of $\left(Z_n\right)_{n \geq 0}$ and $\left(W_n\right)_{n \geq 0}$ 
is available in details in \cite{BaTh2016, Mai_Mar_2017}. Here, we follow same notations as in \cite{BaTh2016}. The following representation is available in the Theorem 2.1 in \cite{BaTh2016}.
\beqn\label{Eq:Grand_representation}
\left(Z_n\right)_{n \ge 0}\el \left(W_n\right)_{n \ge 0}.
\eeqn

\section{Proof of the Main Results} 
\label{Sec:Proofs}

Recall that $\TT_n$ denotes the weighted
RRT with $n+2$ vertices; for convenience we
use $\TT_n$ also to denote its vertex set $\set{o,w_0,w_1,\ldots, w_n}$.
Let $\TT'_n:=\TT_n\setminus\set{o}=\set{w_0,w_1,\ldots, w_n}$,
the set of $n+1$ vertices excluding the root.
Note that the RRT $\TT_n$ is random, but the vertex set is non-random.

For $u,w\in\TT_n$, let $d(u,w)$ denote the graph distance between $u$ and
$w$. 
In particular, $d(o,u)$ is the depth of $u$, which we also denote by $d(u)$.

We begin by proving the following lemma,
where $\rho$ is as in Definition \ref{Duerg}.
\begin{lem}
\label{Lem:Cov_2_colors}
Let $\LL(u,w)$ denote the least common ancestor for the vertices $u,w$ in
the
random recursive tree (RRT).
Given the RRT $\TT_n$,
we have for some suitable constant $C>0$,
\beq\label{Eq:Cov_2_colors}
\Cov\left(W_u=v, W_w=v \mid \TT_n \right) \le C \rho^{\max(d(u, \LL(u,w)), d(w,\LL(u,w))} \le C \rho^{d(u,w)/2}.
\eeq 
\end{lem}

\begin{proof}
Let us denote by $\bP_n$ the conditional probability given the RRT $\TT_n$. By definition, 
\[
\Cov\left(W_u=v, W_w=v \mid \TT_n\right)= \bP_n \left(W_u=v,W_w=v \right)- \bP_n \left(W_u=v\right)\bP_n \left(W_w=v\right).
\]
With $\LL(u,w)$ denoting the least common ancestor between $u$ and $w$, it is easy to see that 
\beqnn
\bP_n \left(W_u=v,W_w=v \right)= \sum_{s \in S} \bP_n \left(W_{\LL(u,w)}=s\right)R^{d(u, \LL(u,w))}(s,v) R^{d(w, \LL(u,w))}(s,v).
\eeqnn

Thus, 
\begin{align}
  \nn &  \Cov\left(W_u=v, W_w=v \mid \TT_n\right)\\
\nn &=\sum_{s \in S}\bP_n \left(W_{\LL(u,w)}=s\right)R^{d(u, \LL(u,w))}(s,v) R^{d(w, \LL(u,w))}(s,v)\\
\nn &\quad-\sum_{s,s' \in S}\bP_n \left(W_{\LL(u,w)}=s\right)\bP_n \left(W_{\LL(u,w)}=s'\right)R^{d(u, \LL(u,w))}(s,v)R^{d(w, \LL(u,w))}(s',v)\\
\nn &=\sum_{s \in S}\bP_n \left(W_{\LL(u,w)}=s\right)R^{d(u, \LL(u,w))}(s,v)\Bigl[R^{d(w, \LL(u,w))}(s,v)\\
\nn& \quad -\sum_{s' \in S}\bP_n \left(W_{\LL(u,w)}=s'\right)R^{d(w, \LL(u,w))}(s',v)\Bigr]\\
\nn &=\sum_{s \in S}\bP_n \left(W_{\LL(u,w)}=s\right)R^{d(u, \LL(u,w))}(s,v)\Bigl[(R^{d(w, \LL(u,w))}(s,v)-\pi_v)\\
\nn& \quad -\sum_{s' \in S}\bP_n \left(W_{\LL(u,w)}=s'\right)(R^{d(w, \LL(u,w))}(s',v)-\pi_v)\Bigr]
.\end{align}
The last equality is obtained by adding and subtracting $\pi_v$ inside the final square bracket.
Recall that we have assumed uniform ergodicity for the Markov chain, so for both $s,s'$ we have 
$|R^{d(w, \LL(u,w))}(s,v)-\pi_v| < C\rho^{d(w, \LL(u,w))}$, which implies that 
\beq
|\Cov\left(W_u=v, W_w=v \mid \TT_n\right)| \le 2C \rho^{d(w, \LL(u,w))}. 
\eeq
The  first inequality in \eqref{Eq:Cov_2_colors} follows by symmetry.
The second inequality 
is obvious as $0<\rho<1$.
\end{proof}

\begin{lem}\label{LA}
  Fix $r$ with $0< r<1$ and define
  \begin{align}\label{laA=}
    A_n=A_n(r)&:=
    \bE \sum_{u \in \TT'_{n}}r^{d(u)},
\\    B_n=B_n(r)&:=
    \bE \sum_{\substack{u,w \in \TT'_{n}}}r^{d(u,w)}. \label{laB=}
  \end{align}
  Then, for some constant $C$ (possibly depending on $r$ and $t$) and all
  $n\ge1$, 
  \begin{align}
    A_n &\le C n^r,\label{laA}\\
    B_n &\le
    \begin{cases}
      C n^{2r}, & \frac12<r<1,\\
            C n\log (n+1), & r=\frac12,\\
            C n, & 0<r<\frac12.\\
          \end{cases}
\label{laB}  \end{align}
\end{lem}
Much more precise asymptotic formulas can be derived by the same method,
but we do not need them.
\begin{proof}
Recall that $w_n$ is the $(n+1)-$th coming vertex,
and assume that $w_n$ is attached to $w\in\TT_{n-1}$. Then, for all
$u\in\TT_{n-1}$,
\begin{align}
d(u, w_n) &= d(u,w)+1.
\end{align}
Hence,
\begin{align}
  A_{n} & =  A_{n-1}+ \bE r^{d(w_n)}\\
\nonumber & = A_{n-1}+ \frac{1}{n+t} \bE \left( \sum_{u \in \TT'_{n-1}}
            r^{d(u)+1}\right)
            +\frac{t}{n+t} r
  \\
  \nn &=\lrpar{1+\frac{r}{n+t}}A_{n-1}+\frac{tr}{n+t}. 
\end{align}
Consequently, by induction and using $A_0=r$,
\begin{align}\label{exactA}
  A_n& = r\sum_{k=0}^n\frac{t}{k+t}\prod_{j=k+1}^n\lrpar{1+\frac{r}{j+t}}
  =
  rt\sum_{k=0}^n\frac{\gG(n+1+t+r)}{\gG(n+1+t)}\frac{\gG(k+t)}{\gG(k+1+t+r)}
.
\end{align}
By standard asymptotics for the Gamma function (following from Stirling's
formula), this yields 
\begin{align}\label{notexactA}
  A_n&
       \le   rt\sum_{k=0}^n C \frac{(n+1)^r}{(k+1)^{r+1}}
         \le C(n+1)^r,
\end{align}
showing \eqref{laA}.

For \eqref{laB} we argue similarly. We have, 
\begin{align}
  \nonumber B_{n}
  & =  B_{n-1}+2 \bE \left( \sum_{u\in \TT'_{n-1}} r^{d(u,w_n)}\right)+1\\
\nonumber & = B_{n-1}+ \frac{2}{n+t} \bE \left( \sum_{u,w \in \TT'_{n-1}}
            r^{d(u,w)+1}\right)
            +\frac{2t}{n+t} \bE \left( \sum_{u\in \TT'_{n-1}} r^{d(u,o)+1}\right)+1
  \\
  \nn&=\lrpar{1+\frac{2r}{n+t}}B_{n-1}+\frac{2rt}{n+t}A_{n-1}+1            
\end{align}
and, with $A_{-1}:=0$,
\begin{align}\label{exactB}
  B_n& = \sum_{k=0}^n\lrpar{1+\frac{2rt}{k+t}A_{k-1}}
       \prod_{j=k+1}^n\lrpar{1+\frac{2r}{j+t}}
.
\end{align}
We use the crude estimate $A_{k-1}\le k$ and estimate the product in
\eqref{exactB} using
Gamma functions as in \eqref{exactA}--\eqref{notexactA}
(with $r$ replaced by $2r$); this yields
\begin{align}
  B_n& \le C \sum_{k=0}^n  \prod_{j=k+1}^n\lrpar{1+\frac{2r}{j+t}}
       \le C \sum_{k=0}^n  \frac{(n+1)^{2r}}{(k+1)^{2r}}
.
\end{align}
This implies \eqref{laB} by a simple summation.
\end{proof}

\subsection{Proof of the Theorem \ref{Thm:a.s.convg_irr}}
\label{SubSec:Proof-of-Theorem}

\begin{proof}

We observe that  the basic recursion ~\eqref{Equ:Fundamental-Recurssion} 
can also be written as
\begin{equation}
\label{recurssion}
U_{n+1}=U_{n} + \chi_{n+1} R 
\end{equation}
where $\chi_{n+1} = \left(\chi_{n+1,v}\right)_{v \in S}$ is such that 
$\chi_{n+1,Z_n}=1$ and $\chi_{n+1,u} = 0$ if $u \neq Z_n$. In other words,
\begin{align}
U_{n+1}=U_n + R_{Z_n}  
\end{align}
where $R_{Z_n}$ is the  $Z_n$-th row of the matrix $R$. 
Hence,
\begin{align}
 U_{n+1}& = U_0 + \displaystyle \sum_{k=1}^{n+1}\chi_k R,\\
  \frac{U_{n+1}-U_0}{n+t+1}&
  = \frac{1}{n+t+1} \displaystyle \sum_{k=1}^{n+1}\chi_k R.
  \label{UsumR}                           
\end{align}
To prove \eqref{Eq:a.s.convg_irr}, it is thus by
\eqref{UsumR},
and 
since $\frac{n+1}{n+t+1} \longrightarrow 1$  as  $n \to \infty$,
enough to show that
\begin{align}
  \frac{1}{n+1} \sum_{k=1}^{n+1}\chi_k R \to \pi
  \qquad \text{ in $\ell^1(S)$, a.s.} 
  \label{Ras}                           
\end{align}
Since $R$ is balanced, the mapping $x\mapsto xR$ is a bounded map
$\ell^1(S)\to\ell^1(S)$, and 
since furthermore $\pi R =\pi$,
to prove \eqref{Ras}
it is enough to show that as $ n \to \infty$,
\begin{equation}\label{SLLN}
  \frac{1}{n+1} \displaystyle \sum_{k=1}^{n+1}\chi_{k}\longrightarrow \pi
  \qquad \text{ in $\ell^1(S)$, a.s.} 
\end{equation} 
Both sides of \eqref{SLLN} can be regarded as probability distributions on
$S$, and therefore, the convergence in $\ell^1$ is equivalent to convergence
of every coordinate, i.e., to
\begin{equation}\label{Eq:SLLN}
  \frac{1}{n+1} \sum_{k=1}^{n+1}\chi_{k,v}\longrightarrow \pi_v
  \qquad\text{a.s.\ for every }v \in S.
\end{equation} 
Moreover, 
\eqref{Eq:local_time_color_almost_sure}
is just another way to write \eqref{Eq:SLLN}.
Hence, to show the theorem, it suffices to show \eqref{Eq:SLLN}.

Recall that $\chi_{k,v}=\bone_{\{Z_{k-1}=v\}}.$
From Theorem 3.3(a) of \cite{BaTh2016} (which easily is extended to general
$t$), it follows that
\begin{equation}\label{Eq:Expectation for random color}
\frac{1}{n+1}\bE \left[\displaystyle \sum_{k=1}^{n+1}\chi_{k,v}\right]= \frac{1}{n+1}\displaystyle \sum_{k=0}^{n}\bP\left(Z_k=v\right) \longrightarrow \pi_v \text{ as } n \to \infty.
\end{equation}

Note that
$
\left|\bone_{\{Z_k=v\}}-\bE \bone_{\{Z_k=v\}}\right|
\leq 1
$. 
Therefore, from the Strong Law of Large Numbers for correlated random
variables \cite[Theorem 1]{Ly88},
it follows that, if we prove 
\begin{equation}\label{Eq:Summability of var}
\displaystyle \sum_{n \geq 0} \frac{1}{n+1}\mathrm{Var} \left(\frac{1}{n+1}\displaystyle \sum_{k=0}^{n}\bone_{\{Z_k=v\}}\right)<
\infty,
\end{equation}
then, as $n \to \infty$,
\begin{align}
    \frac{1}{n+1} \sum_{k=1}^{n+1}\chi_{k,v}
=
\frac{1}{n+1}\displaystyle \sum_{k=0}^{n}\bone_{\{Z_k=v\}} \longrightarrow \pi_v \text{ a.s.},
\end{align} 
which will complete the proof.
In other words, if we define
\begin{equation}\label{Jnv}
J_{n,v}:=\mathrm{Var} \left( \sum_{k=0}^{n}\bone_{\{Z_k=v\}}\right),
\end{equation}
then, it suffices to show that
\begin{align}\label{abab}
  \sum_{n=1}^\infty \frac{1}{n^3}J_{n,v}<\infty.
\end{align}

Now, recalling \eqref{Eq:Grand_representation}, \eqref{Jnv} can be expanded
as
\begin{align}\label{Jnv2}
J_{n,v}&=\mathrm{Var} \left( \sum_{k=0}^{n}\bone_{\{W_k=v\}}\right)
= \sum_{u,w\in\TT'_n}\Cov \left(\bone_{\{W_u=v\}},\bone_{\{W_w=v\}}\right). 
\end{align}
 We use the conditional covariance formula to get
\begin{align}
\Cov \left(\bone_{\{W_u=v\}},\bone_{\{W_w=v\}}\right)
 &=  \bE\left[\Cov \left(\bone_{\{W_u=v\}},
                       \bone_{\{W_w=v\}} \mid \TT_{n}\right)\right] \nonumber \\
    & \qquad +\Cov \bigpar{\bE\left[\bone_{\{W_u=v\}}\mid \TT_{n}\right], \,
                       \bE\left[\bone_{\{W_w=v\}}\mid \TT_{n}\right]}
. \label{Eq:cond_cov}
\end{align}
Now, using Lemma \ref{Lem:Cov_2_colors}, we obtain 
\[
\Cov \left(\bone_{\{W_u=v\}}, \, \bone_{\{W_w=v\}} \mid \TT_{n} \right)
\le C \rho^{\frac{d(u,w)}{2}},
\]
where $d(u,w)$ 
denotes the graph distance between $u$ and $w$, and $C$ is a suitable
positive constant. Therefore, from \eqref{Jnv2}--\eqref{Eq:cond_cov},
the contribution to $J_{n,v}$
from the first part of
\eqref{Eq:cond_cov} is at most
\beq\label{Eq:sum_cov_part_1}
C\bE \left( \sum_{\substack{u,w \in \TT'_{n}}} \rho^{d(u,w)/2} \right)
=CB\bigpar{\rho^{1/2}},
\eeq
where we recall \eqref{laB=} and take  $r:=\rho^{1/2}$.


For the second part on the RHS of \eqref{Eq:cond_cov}, we have that, given
$\TT_n$, the distribution of $W_u$ is $(U_0/t)R^{d(u)}$. Hence,
\begin{align}
  \bE\lrpar{\bone_{\{W_u=v\}}\mid\TT_n}=(U_0/t)R^{d(u)}(v),
\end{align}
and thus it follows from the uniform ergodicity assumption
\eqref{Eq:GeoErforR} that 
\begin{align}
 \bigabs{\bE\lrpar{\bone_{\{W_u=v\}}\mid\TT_n}-\pi_v}\le C\rho^{d(u)}.
\end{align}
Consequently,
\begin{align}
\nn&\sum_{u,w\in\TT'_n}\Cov \bigpar{\bE\left[\bone_{\{W_u=v\}}\mid \TT_{n}\right], \,
                       \bE\left[\bone_{\{W_w=v\}}\mid \TT_{n}\right]}
         = \Var\lrpar{\sum_{u\in\TT'_n}\bE\left[\bone_{\{W_u=v\}}\mid \TT_{n}\right]}
     \\\nn&\hskip2em
  = \Var\lrpar{\sum_{u\in\TT'_n}\lrpar{\bE\left[\bone_{\{W_u=v\}}\mid
            \TT_{n}\right]-\pi_v}}
    \le   \bE\lrpar{\sum_{u\in\TT'_n}
          \lrpar{\bE\left[\bone_{\{W_u=v\}}\mid \TT_{n}\right]-\pi_v}}^2
  \\\nn&\hskip2em
         \le   \bE\lrpar{C\sum_{u\in\TT'_n}\rho^{d(u)}}^2
         =  C \bE \sum_{u,w\in\TT'_n}\rho^{d(u)+d(w)}
          \le  C \bE \sum_{u,w\in\TT'_n}\rho^{d(u,w)}
 \\\nn&\hskip2em
                   = C B_n(\rho),
\end{align}
where we use the fact that $d(u)+d(w) \ge d(u,w)$ and $0< \rho<1$ to obtain
the last inequality.
Hence,
the contribution to $J_{n,v}$
from the second part of
\eqref{Eq:cond_cov} is at most $CB_n(\rho)$.

Combining the contributions from the two parts
of \eqref{Eq:cond_cov},
we thus have shown that, recalling $0<\rho<1$, 
\begin{align}
  J_{n,v}\le C B_n\bigpar{\rhoqq} + CB_n(\rho)
  \le C B_n\bigpar{\rhoqq}.
\end{align}
Hence, 
we can use Lemma \ref{LA} and conclude
\eqref{abab},
which completes the proof.
\end{proof}

\section{Random walk with linear reinforcement on the star graph}
\label{Sec:Appl}
In this section, we consider a linearly reinforced random walk model on the 
infinite (countable) \emph{star graph}.
We will show that the almost sure convergence for the local times for this walk can be derived using our main result stated in Section ~\ref{SubSec:Result}.

Let us consider a special type of vertex-reinforced nearest neighbor random
walk $(X_n)_{n \geq 0}$ on an infinite star graph, with a loop at the
root. We denote the root by $v_0$ and the other vertices by $v_i, i \ge 1$.
Each edge is regarded as a pair of directed edges in opposite directions;
the notation $(v_i,v_j)$ indicates that the
edge is from $v_i$ to $v_j$.
We impose the following condition on the walk
that $v_0$ is a special vertex, in the sense that, whenever the walker takes
the edge $(v_j,v_0)$, for any $j$, it puts an additional weight of
$\alpha_j:=(\alpha_{j,i})_{i \ge 0}$ on  the vertices, such that,
$\sum_{i}\alpha_{j,i}< \infty$. If the edge taken is $(v_0, v_j)$, $j \neq
0$ then no vertex is reinforced.

Initially, $X_0 \equiv v_0$, the walker is at the root, and jumps to one of
the adjacent neighbors with probability proportional to the given weights
$\de_i$, such that, $\de:=\sum_{i \ge 0}\de_i < \infty.$ At any time $n\ge
1$, the transition probabilities for the random walk is governed by 
\beqna\label{Eq:Trans_VRRW}
\bP \left(X_{n+1}=v_j | X_n=v_i\right)= \begin{cases}
\frac{\De_{n,j}}{\sum_{k}\De_{n,k}}, \text{ when } i=0,\\
\bone_{\{j=0\}}, \text{ for } i \ge 1.
\end{cases}
\eeqna
where $\De_{n,j}$ denotes the weight at the vertex $v_j$ at time $n$.

Observe that, if we denote by $\sig_k$, the random time at which the weights are updated for the $k$-th time, then $\sig_{k+1}= \sig_k+Y_{k+1}$, where $Y_{k+1} \in \{1,2\}$ is a random variable, such that, 
$$
\bP\left(Y_{k+1}=1| \De_0, \De_{\sig_1}, \ldots \De_{\sig_k}\right) = \frac{\De_{\sig_k,0}}{\sum_j \De_{\sig_k,j}}.
$$
Therefore, the weight sequence at these updating random times can be coupled with an infinite color urn model, as described below. 
 
Consider the urn model with colors indexed by $S:=\{0,1,2, \ldots\}$, and an
initial composition $U_0=(\de_i)_{ i \ge 0}$. The replacement matrix is such
that the $j$-th row of the matrix is $\alpha_j$.
Since the graph
is a star graph, for the random walk  to take a step along $(v_i,v_0)$, $i
\neq 0$, it implies that the walker has jumped along the edge $(v_0, v_i)$
according to the transition probabilities given by
\eqref{Eq:Trans_VRRW}. So, if we consider the sequence of weights at time
$\sig_1, \sig_2, \ldots$, then the processes are coupled such that
\beq\label{Eq:law_walk_urn}
\left(\Delta_{\sig_n}\right)_{ n \ge 0} = \left(U_n\right)_{n \ge 0}.
\eeq

In particular, if $\sum_{i} \alpha_{j,i}=1$ for each $j$, then the
replacement matrix is a stochastic matrix. Henceforth, we assume that
$\alpha_j$ is a probability vector for every $j \ge 0.$ We also assume that
$\alpha_j$ are such that, the Markov chain corresponding to the replacement
matrix is irreducible, aperiodic and uniformly ergodic. 

A particular example of such a matrix is when $\alpha_0=(p_j)_{j \ge 0}$,
with $p_j>0$ and $\sum_j p_j=1$, 
and, for $j \neq 0$,
$\alpha_{j,i}=1 \text{ if } i =0$, and $0$ otherwise.
(Our conditions, including uniform ergodicity, are easily verified.)

\begin{theorem}\label{Thm:cgs_distribution_update_times_VRRW_star}
Let $X_n$ be a vertex reinforced random walk on an infinite star graph with
a loop at the root, such that the
replacement matrix is an irreducible, aperiodic and uniformly ergodic
stochastic matrix.
Let the transition probabilities of $X_n$ be as in
\eqref{Eq:Trans_VRRW}. If we denote by $\sig_n$, the $n$-th update time,
then as $n \to \infty$,
\beq \label{Eq:cgs_distribution_update_time}
\frac{\sig_n}{n+1} \longrightarrow 2- \pi_0, \text{ a.s.\ and in } L_1.
\eeq
Furthermore, for any $j \ge 0$, as $n \to \infty$
\beq\label{Eq:cgs_distribution_weights}
\frac{\De_{n,j}}{n+\de} \longrightarrow \frac{\pi_j}{2-\pi_0}, \text{  a.s.,}
\eeq
where $\pi$ is the stationary distribution of the coupled urn process as
defined in \eqref{Eq:law_walk_urn}.
\end{theorem}

\begin{proof}
As observed earlier $\sig_{k+1}=\sig_k+Y_{k+1}$, where $Y_{k+1} \in \{1,2\}$, and 
$$
\bP\left(Y_{k+1}=1| \De_0, \De_{\sig_1}, \ldots \De_{\sig_k}\right) = \frac{\De_{\sig_k,0}}{\sum_j \De_{\sig_k,j}}.
$$
Let us denote by $\wtilde{\sig}_k:= \sum_{i=0}^k \bone_{\{Y_j=1\}}.$ Then
from the conditional distribution of $Y_k$ above and from
\eqref{Eq:law_walk_urn}, we have, using the coupling above,
\beq \label{Eq:law_update_by_one}
\wtilde{\sig}_n = \sum_{k=0}^n \bone_{\{Z_k=0\}}=N_{n,0},
\eeq
 where $Z_k$ denotes the random color of the ball selected in the coupled
 urn model.
 From Theorem  \ref{Thm:a.s.convg_irr}\ref{local_times_color},
 we know that as $n \to \infty$, 
 $N_{n,0}/(n+1) \longrightarrow \pi_0$
 a.s. Thus, as $n \to \infty$, 
\beq \label{Eq:convergence_distribution_update_by_one}
\frac{\wtilde{\sig}_n}{n+1} \longrightarrow \pi_0, \text{ a.s.} 
\eeq
Since $\sig_n= \wtilde{\sig}_n+ 2(n-\wtilde{\sig}_n)$, so
\eqref{Eq:cgs_distribution_update_time} follows immediately.
Since $0\le \frac{\sig_n}{n+1}\le 1$, the $L_1$ convergence in \eqref{Eq:cgs_distribution_update_time} follows by dominated convergence theorem. 

Denote by $m(n):=\sup\{k: \sig_k \le n\}$.
Then it follows from
\eqref{Eq:cgs_distribution_update_time} that,
as $n \to \infty$,
$\frac{m(n)}{n+1} \longrightarrow \frac{1}{2- \pi_0}$, a.s.
From
\eqref{Eq:law_walk_urn} and
Theorem \ref{Thm:a.s.convg_irr}\ref{convg_irr}, we have as $n \to\infty$,
\begin{align}\label{Eq:Convergence_weight_sequence}
  \frac{\De_{n,j}}{n+\de}
  = \frac{\De_{\sigma_{m(n)},j}}{m(n)} \frac{m(n)}{n+\de}
    = \frac{U_{m(n),j}}{m(n)} \frac{m(n)}{n+\de}
  \longrightarrow \frac{\pi_j}{2-\pi_0}, \quad\text{a.s.}
\end{align}
\end{proof}

\section*{Acknowledgement}
The authors are grateful to Colin Desmarais and Cecilia Holmgren for various discussions they had with them over the course of working on this paper. We are grateful for their insightful ideas and remarks, which helped us improve the quality of this paper.

\bibliographystyle{plain}

\bibliography{RT}

\end{document}